\documentclass[12pt, reqno]{amsart}

 \usepackage{amsmath}
 \usepackage{amsfonts}
 \usepackage{amssymb}
 \usepackage{times}
 \usepackage{graphicx}
 \usepackage[justification=centering]{caption}
 \usepackage{subfigure}

\newtheorem{thm}{Theorem}[section]
\newtheorem{prop}[thm]{Proposition}
\newtheorem{cor}[thm]{Corollary}

\theoremstyle{definition}

\numberwithin{equation}{section}

\newtheorem*{mthm}{Main Theorem}

\newtheorem{fg*}{figure}
\newtheorem{remr}{Remark}
\newtheorem{exe}{Example}

\newtheorem*{thmz}{Theorem Z}
\newtheorem*{thmk}{Theorem KRL}

\newtheorem{thmc}{Conjecture }

\newcommand{\C}{{\mathbb C}}
\newcommand{\D}{{\mathbb D}}

\newcommand{\ind}{\int_\D}

\newcommand{\diri}{\mathfrak{D}}

\begin{document}

\title[composition operators with closed ranges  ]
{composition operators with closed range on the Dirichlet space}

\author{Guangfu Cao}
\address{Cao: School of Mathematics and Information Science,
Guangzhou University, Guangzhou 510006, China.}
\email{guangfucao@163.com}

\author{Li He*}
\address{He: School of Mathematics and Information Science,
Guangzhou University, Guangzhou 510006, China.}
\email{helichangsha1986@163.com}

\thanks{G.C. was supported by NNSF of China (Grant No. 12071155), and
L.H. was supported by NNSF of China (Grant No. 11871170).}
\thanks{*Corresponding author, email: helichangsha1986@163.com}

\begin{abstract}
It is well known that the composition operator on Hardy or Bergman space has a closed range  if and only if its Navanlinna counting function induces a reverse Carleson measure. Similar conclusion is not available on the Dirichlet space. Specifically, the reverse Carleson measure is not enough to ensure that the range of the corresponding composition operator is closed.   However,  under certain assumptions, we in this paper set the necessary and sufficient condition for a composition operator on the Dirichlet space to have  closed range. \end{abstract}

\maketitle
\section{Introduction}
  Let  $\D$ be the unit disc in the complex plane $\C, $ $\mathbb{T} $ be the unit circle, and let $dA$ denote
area measure on $\D$. The Dirichlet space, denoted by $\diri (\D)$, consists of analytic
functions $f$ on $\D$ such that
$$\|f\|^2=|f(0)|^2+\ind|f'(z)|^2\,dA(z)<\infty.$$
Obviously, $\diri (\D)$ is a Hilbert space with the inner product 
$$\langle f,g\rangle=f(0)\overline{g(0)}+\ind f'(z)\overline{g'(z)}\,dA(z).$$
 Set
$e_0(z)=1,\quad e_n(z)=\frac1{\sqrt\pi}\frac{z^n}{\sqrt n},\qquad (n=1,2,,\cdots),$ then $\{e_n\}$ is the standard orthonormal basis of $\diri (\D).$ 
It follows that the reproducing kernel of  $\diri (\D)$ is
$$K(z,w)=\sum_{n=0}^\infty e_n(z)\overline{e_n(w)}=1+\frac1{\pi}
\log\frac1{1-z\overline w}.$$
  Given an analytic self-map $\varphi:\D\to\D$, for any $w\in \varphi (\D),$  write $
 n_{\varphi}(w)
 $ is the cardinality of the set $\varphi^{-1}(w),$   $N_{\varphi }(w)$ is the (Nevanlinna) counting number function of $\varphi ,$ and
$$
 \tau_{\varphi}(w)=\frac{N_{\varphi }(w)}{log\frac{1}{|w|}}.
 $$
 
 In the past few decades, the closed-range composition operators on various  function spaces  has attracted wide attention,  we refer to \cite{AF, CTW, GYZ, GZZ, KW, KK, L, L1, M, P, Y, Zor, Zo, Zo1} and the references therein.
 In 1974 Cima, Thompson, and Wogen determined when the composition operators on Hardy space $H^2$ have  closed-ranges by utilizing  the boundary behaviour of the inducing maps (see \cite{CTW}). In the end of their paper,  they posed  the problem of characterizing the closed-range composition operators using the properties of the range of the inducing maps on the unit disc rather than the properties of the boundary.  Zorboska \cite{Zo}  gave an answer  to this question by using the Nevanlinna counting function and Luecking's measure theoretic results on inequalities on Bergman spaces. He proved the following
  \begin{thmz}(\cite{Zo}). Let $\varphi $ be an  analytic self-mapping of $\D.$  Assume that $C_{\varphi }$ is a bounded composition operator on the Hardy space $
 H^2(\D)$ defined as
 $$
 C_{\varphi }f(z)=(f\circ \varphi )(z), \hskip 5mm\forall \hskip 2mm f\in H^2(\D).
 $$
Then  $C_{\varphi }$ has  closed range if and
only if there exists a positive constant $c > 0$ such that the set $G_c = \{z \in \D: \tau_{\varphi }(z) > c\}$ satisfies the condition: 

There exists a constant $\delta > 0$ such that 
$$
A(G_c \cap S(\zeta , r))>\delta A(\D \cap S(\zeta ,r))\hskip 10mm (*)
$$
for all $\zeta \in \mathbb{T}$ and $r > 0,$ where $S(\zeta, r)=\{z\in \D:|z-\zeta|<r\}.$ 
 \end{thmz}

In addition, Zorboska constructed a counterexample which shows that the range of $C_{\varphi }$ may not be closed even if $\varphi (\D)=\D.$  It seems a little strange  since   $C_{\varphi }$ is invertible on $H^{2}(\D)$ if $\varphi \in Aut(\D).$

For the case of the (weighted) Bergman space, one  also obtained the necessary and sufficient conditions for the closed range  composition operators (see \cite{AF, L}). However, in the case of the Dirichlet space, it  seems rather difficult.

Recall that the pseudo-hyperbolic metric on $\D$  is defined by 
 $$
 \rho (z,w)=|\frac{z-w}{1-\bar{z}w}|, \hskip 5mm z, w\in \D,
 $$
 and the Bergman metric is defined by
 $$
 \beta(z, w)=\frac{1}{2}log\frac{1+\rho(z,w)}{1-\rho(z,w)}, \hskip 5mm z, w\in \D.
 $$
  For $0 < \eta  < 1,$ $z\in \D$ and $r>0,$ 
 write 
 $$
 D_{\eta }(z)=\{w\in \D|\rho (z,w)< \eta \}
 $$
and
 $$
 D(z, r)=\{w\in \D|\beta(z,w)< r\}.
 $$
 For $G\subset \mathbb{C},$ $A(G)$ denotes the Lebesgue measure of $G.$ Then for any fixed positive $ r,$ we have
 $$
 A(D(z, r))\sim (1-|z|^{2})^{2}.
 $$
 In fact,  $D(z, r)$ is the Euclidean disk with Euclidean center $(1-s^2)z/(1-s^2|z|^2)$ and Euclidean radius $(1-|z|^2)s/(1-s^2|z|^2),$ where $s=tanh(r)\in (0,1).$
 One may consult the book  \cite{Z} by Kehe Zhu for details. 
 
Luecking \cite{L3} proved that a necessary condition for a composition operator $C_{\varphi }$ on the Dirichlet space $\mathcal{D}(\D)$ to have closed range is that  $n_{\varphi}dA$ must be a reverse Carleson measure, that is, there is a $\delta >0$ and $\eta \in(0, 1)$ such that
 $$
 \int_{\varphi(\D )\cap D_\eta(z)}n_{\varphi }(w)dA(w)> \delta A(D_{\eta }(z ))
 $$
  for all $z \in \mathbb{D}, $ or equivalently,
   there is a $\delta >0$ such that
 $$
 \int_{\varphi(\D )\cap S(\zeta ,r )}n_{\varphi }(w)dA(w)> \delta A(S(\zeta ,r ))
 $$
 for all $\zeta \in \mathbb{T}.$ 
 Note that
 $$
 D(z,r)=\{w\in \D: \beta (z, w)<r\}=D_{\eta }(z)=\{w\in \D: \rho (z, w)<\eta \}
 $$
 for $\eta =\frac{e^{2r}-1}{e^{2r}+1},$ and then that the following statements are equivalent:
 
 (1) There is a $\delta >0$ and $r>0$ such that
  $$
 \int_{\varphi(\D )\cap D(z ,r )}n_{\varphi }(w)dA(w)> \delta A(D(z ,r ))
 $$
 for all $z \in \mathbb{D}.$ 

(2) There is a $\delta >0$ and $\eta \in(0, 1)$ such that
 $$
 \int_{\varphi(\D )\cap D_{\eta }(z )}n_{\varphi }(w)dA(w)> \delta A(D_{\eta}(z))
 $$
  for all $z \in \mathbb{D}.$ 
  
In 1999 Luecking \cite{L1} constructed a self-mapping  $\varphi $ on $\D$  such that $n_{\varphi}dA$ is a reverse Carleson measure, but $C_{\varphi }$ has not closed range.  Unfortunately, there is no  necessary and sufficient condition for a composition operator to have  closed range. In contrast to Hardy space, the composition operator with surjective symbol must have closed range on Dirichlet space. In fact, if $\varphi $ is an analytic self-mapping  on $\D$, and  $\varphi (\D)=\D,$ then $n_{\varphi}(w)\geq 1$ for any $w\in \D,$ and 
\begin{eqnarray*}
\|C_{\varphi }f\|^2_{\mathcal{D}}&=&\int_{\D}|(f\circ \varphi )'(z)|^2dA(z)\\
&=&\int_{\varphi (\D)}|f'(w)|^2n_{\varphi}(w)dA(w)\\
&=&\int_{\D}|f'(w)|^2n_{\varphi}(w)dA(w)\\
&\geq &\int_{\D}|f'(w)|^2dA(w)=\|f\|^{2}_{\mathcal{D}}.
\end{eqnarray*}
This implies that $C_{\varphi }$ has  closed range.

The discussions above means that even if the measure  $n_{\varphi}(w)dA(w)$ induced by counting function $n_{\varphi}(w)$ is a  reverse  Carleson measure, it also cannot guarantee that the composition operator $C_{\varphi}$ has  closed range.  However, if $\varphi$  is surjective, then the range of $C_{\varphi}$ must be closed. This phenomenon is not  surprising, since the counting function $n_{\varphi}(w)$  may be unbounded when $w$ approaches  the boundary of the domain, although the image of the symbol map cannot fill  the neighborhood of $\mathbb{T},$  the counting function may introduce a reverse Carleson measure. It can show up from the counterexample of Luecking in \cite{Zo}. This shows that the reverse Carleson measure is not  a proper condition for composition operator with closed range on  Dirichlet space $\mathcal{D}(\D)$ in some extreme cases.





Notice that 
$$
\int_{\varphi (\D)\cap  D(z ,r)}n_{\varphi}(w)A(w)\geq \int_{\varphi (\D)\cap  D(z ,r)}A(w)=A(\varphi (\D)\cap  D(z ,r)),
$$
then the inequality
$$
A(\varphi (\D)\cap  D(z ,r))\geq \delta A(D(z ,r))
$$
for all $z \in \mathbb{D}$ and some $r>0$ implies that $n_{\varphi}(w)A(w)$ is a reverse Carleson measure, but the contrary may not be true. In other words, the inequality 
$$
A(\varphi (\D)\cap  D(z ,r))\geq \delta A(D(z ,r))
$$
is stronger than that $n_{\varphi}(w)dA(w)$ is a reverse Carleson measure. Further, they are equivalent under the assumption that $n_{\varphi }(w)$ is bounded. In fact,
if $n_{\varphi}(w)$ is bounded on $\D,$ then for any $z \in \mathbb{D}$ and $r>0,$ there is a constant $M>0$ such that
\begin{eqnarray*}
MA(\varphi(\D)\cap D(z ,r))&\geq &\int_{\varphi (\D)\cap  D(z ,r)}n_{\varphi}(w)dA(w)\\
&\geq & \int_{\varphi (\D)\cap  D(z ,r)}dA(w)\\
&=&A(\varphi(\D)\cap D(z ,r)).
\end{eqnarray*}
In this case, the inequality
$$
A(\varphi (\D)\cap  D(z ,r))\geq \delta A(D(z ,r))
$$ is equivalent to that $n_{\varphi}(w)dA(w)$ is a reverse Carleson measure. 
It seems that a proper condition for composition operator with closed range is:

There exists a constant $\delta > 0$ and $r>0$ such that 
$$
A(\varphi(\D) \cap D(z , r))\geq \delta A(D(z ,r))\hskip 10mm (**) 
$$
for all $z \in \mathbb{D}.$  

Under some natural assumptions,   $(**)$ is indeed the necessary and sufficient condition for composition operators to have  closed range.
Our main result is as follow
 \begin{mthm}Let $\varphi $ be an  analytic self-mapping function of $\D.$  If
$$
lim_{k\rightarrow \infty }sup_{f\in (\mathcal{D}(\D))_1}\int_{\varphi(\D)\cap \{w\in \D:n_{\varphi }(w)>k\}}|f'(w)|^2n_{\varphi }(w)dA(w)=0,
$$
where $(\mathcal{D}(\D))_1$ denotes the unit sphere of $\mathcal{D}(\D),$ then the following statements are equivalent:
 \begin{enumerate}
\item[{\bf(a)}] 
 $R( C_{\varphi })$  is closed on $\mathcal{D}(\D).$
  \item[{\bf(b)}] There is a $\delta >0$ and $r>0$ such that
 $$
  A(\varphi(\D)\cap D(z,r) )\geq \delta A(D(z,r))
$$
 for all  $z\in \D$.
 \item[{\bf(c)}]  $n_{\varphi }(w)dA(w)$ is a reverse Carleson measure.
  \end{enumerate}
\end{mthm}

 \section{Composition operator with closed range}

 
 \begin{prop}Let $\varphi $ be an  analytic self-mapping function of $\D.$  Assume $C_{\varphi }$ is a bounded composition operator on $
 \mathcal D(\D)$ defined as
 $$
 C_{\varphi }f(z)=(f\circ \varphi )(z), \hskip 5mm\forall \hskip 2mm f\in \mathcal D(\D).
 $$
 If $R( C_{\varphi })=\{(f\circ \varphi |f\in \mathcal{D}(\D)\}$, the range of $ C_{\varphi },$ is closed, then  for any $\zeta \in \mathbb{T}$ and any $r >0,$
 $$
 S(\zeta ,r )\cap \varphi (\D)\neq \phi .
 $$
 Equivalently, $\mathbb{T} \subset \overline{\varphi (\D)}.$
  \end{prop}
  \begin{proof} Assume $R( C_{\varphi })$ is closed, since $kerC_{\varphi }=\{0\},$ we know that there is a constant $c>0$ such that  
 $$
\| C_{\varphi }f\|\geq c\|f\|, \hskip 5mm \forall f\in \mathcal D(\D).\hskip 20mm (***)
$$
Thus
\begin{eqnarray*}
\| C_{\varphi }f\|^2&=&\int_{\D}|(f\circ \varphi )'(z)|^2 dA(z)+|f\circ \varphi(0)|^2\\
&=&\int_{\varphi (\D)}|f'(w)|^2 n_{\varphi }(w)dA(w)+|f\circ \varphi(0)|^2\\
&\geq &c^2[ \int_{\D}|f'(z)|^2 dA(z)+|f(0)|^2].
\end{eqnarray*}
 If  there is a $\zeta \in  \mathbb{T}$ and  $r >0$ such that
 $$
 S(\zeta ,r)\cap \varphi (\D)= \phi ,
 $$
  let
$$
f_{\zeta }(z)=\frac{1+\bar{\zeta}\cdot z}{2}, \hskip 5mm f_{k}(z)=f_{\zeta }^{k}(z),
$$
then $f_{\zeta }(z)$ is the peak function at $\zeta $ on $\D.$ For arbitrary open neighborhood $U$ of $\zeta ,$ it is easy to see that
\begin{eqnarray*}
&&lim_{k\rightarrow \infty }\sqrt[k]{\frac{\int_{\D-U}|f'_{k}(z)|^2dA(z)}{\int_{\D}|f'_{k}(z)|^2dA(z)}}\\
&=&\frac{max|f_{\zeta}|_{\D-U}|}{max|f_{\zeta}|_{\D}|}=max|f_{\zeta}|_{\D-U}|<1.
\end{eqnarray*}
Hence 
$$
\frac{\int_{\D-U}|f'_{k}(z)|^2dA(z)}{\int_{\D}|f'_{k}(z)|^2dA(z)}\rightarrow 0 \hskip 5mm\mbox{as}\hskip 5mm k\rightarrow \infty . 
$$
Direct calculation gives that
\begin{eqnarray*}
\frac{\| C_{\varphi }f_{k}\|^2}{\|f_{k}\|^2}&=&\frac{\int_{\D}|(f_{k}\circ \varphi )'(z)|^2 dA(z)+|f_{k}\circ \varphi(0)|^2}{\int_{\D}|f'_{k}(z)|^2dA(z)+|f_{k}(0)|^2}\\
&=&\frac{\int_{\varphi (\D)}|f'_{k}(w)|^2 n_{\varphi }(w)dA(w)+|f_{k}\circ \varphi(0)|^2}{\int_{\D}|f'_{k}(z)|^2dA(z)+|f_{k}(0)|^2}.
\end{eqnarray*}
 By the boundedness of $C_{\varphi },$   there exists a positive constant $M$ such that   
$$
\int_{\D}|f'(w)|^2 n_{\varphi }(w)dA(w)\leq M \int_{\D}|f'(w)|^2 dA(w), \hskip 5mm \forall f\in \mathcal{D}(\D).
$$ 
In particular,
$$
\int_{\D} n_{\varphi }(w)dA(w)\leq M \int_{\D}dA(w)=MA(\D).
$$
Thus
$$
lim_{k\rightarrow \infty }\sqrt[k]{\int_{\varphi(\D)}|f'_{k}(z)|^2n_{\varphi}(z)dA(z)}=max|f_{\zeta}|_{\varphi(\D)}|.
$$


Since there is a neighborhood $U$ of $\zeta $ such that $U\cap \overline{\varphi (\D)}=\phi ,$ we have  $ \overline{\varphi (\D)}\subset \overline{\D}-U.$ Further,
\begin{eqnarray*}
&&lim_{k\rightarrow \infty }\sqrt[k]{\frac{\int_{\varphi(\D)}|f'_{k}(z)|^2n_{\varphi}(z)dA(z)}{\int_{\D}|f'_{k}(z)|^2dA(z)}}\\
&=&\frac{max|f_{\zeta}|_{\varphi(\D)}|}{max|f_{\zeta}|_{\D}|}=max|f_{\zeta}|_{\varphi(\D)}|<1.
\end{eqnarray*}
With the fact that
$
f_{k}(0)\rightarrow 0$ and $f_{k}\circ \varphi(0)\rightarrow 0,
$
we get
$$
\frac{\| C_{\varphi }f_{k}\|^2}{\|f_{k}\|^2}=\frac{\int_{\varphi (\D)}|f'_{k}(w)|^2n_{\varphi}(z) dA(w)+|f_{k}\circ \varphi(0)|^2}{\int_{\D}|f'_{k}(z)|^2dA(z)+|f_{k}(0)|^2}\rightarrow 0,\hskip3mm \mbox{as}\hskip 3mm k\rightarrow \infty .
$$
This contradicts to $(***).$ It shows that 
$$
 S(\zeta ,r )\cap \varphi (\D)\neq \phi ,
 $$
 which implies $\mathbb{T} \subset \overline{\varphi (\D)}.$
\end{proof}
Let $\alpha>-1$ and $1\leq p<\infty .$ We say that $G$, a Borel subset of $\D,$ satisfies
the reverse Carleson condition on $A_{\alpha}^{p},$  the weighted Bergman space, if there exists positive constant $\eta $ such
that  
$$
\eta \int_{G} |f(z)|^p(1-|z|^2)^{\alpha }dA(z) \geq \int_{\D}|f(z)|^p(1- |z|^2)^{\alpha }dA(z), \hskip4mm \forall f \in  A_{\alpha}^{p}.
$$
 \begin{thmk} (\cite {KK, L2}) Let $G\subset \D$ be a Borel subset of $\D.$ Then
 $ G$ satisfies the reverse Carleson condition if and only if the following condition holds:\\
There is a $\delta >0$ and $r>0$ such that 
$$
A(G\cap D(z, r)) > \delta A(D(z, r)) \hskip 5mm \mbox{for each} \hskip 5mm z \in \D,
$$ 
or equivalently, there is a $\delta >0$ and $\eta \in (0, 1)$ such that 
$$
A(G\cap D_{\eta} (z)) > \delta A(D_{\eta }(z)) \hskip 5mm \mbox{for each} \hskip 5mm z \in \D.
$$ 
\end{thmk}
  \begin{prop}
 Let $\varphi $ be an  analytic self-mapping of  $\D$  and  $C_{\varphi }$ be bounded on $\mathcal{D}(\D).$ Then the following statements are equivalent:
 \begin{enumerate}
  \item[{\bf(a)}] There is a $\delta >0$ and $r>0$ such that
 $$
 A(\varphi(\D)\cap D(z,r) )\geq \delta A(D(z,r))
$$
 for all  $z\in \D$.
  \item[{\bf(b)}] For arbitrary  $\alpha \in (0, 1)$, there is a   $\delta >0$  and $r>0$ such that
  $$
  \int_{\varphi(\D)\cap D(z, r)}n_{\varphi}^{\alpha}(w)dA(w)\geq \delta A(D(z, r))
$$
 for all  $z\in  \D.$ 
 \item[{\bf(c)}] For arbitrary $\alpha \in (0, 1)$, there is a $\delta >0$  such that
 $$
 \int_{\varphi(\D)}|f'(w)|^2n_{\varphi}^{\alpha}(w)dA(w)\geq \delta \int_{\D}|f'(w)|^2dA(w)
$$
 for all  $f\in \mathcal{D}(\D).$
  \end{enumerate}
 \end{prop}
\begin{proof} Note $n_{\varphi}^{\alpha}(w)\geq 1$ for any $w\in \varphi(\D)$ and $\alpha\in (0, 1),$ we have
$$
 \int_{\varphi(\D)}|f'(w)|^2n_{\varphi}^{\alpha}(w)dA(w)\geq \int_{\varphi (\D)}|f'(w)|^2dA(w)
 $$
for all $f\in  \mathcal{D}(\D).$  Hence $(a)\Longrightarrow(c)$ is obvious by Theorem KRL.

Conversely, for arbitrary $\alpha \in (0, 1)$ and $\epsilon >0,$ there is a $N\in \mathbb{N}$ such that 
$$
\frac{n^{\alpha }}{n}=\frac{1}{n^{1-\alpha}}<\epsilon 
$$ 
for all $n\geq N.$ Thus 
\begin{eqnarray*}
 &&\int_{\varphi(\D)\cap \{w: n_{\varphi }(w)>N\}}|f'(w)|^2n_{\varphi}^{\alpha}(w)dA(w)\\
 &=& \int_{\varphi(\D)\cap \{w: n_{\varphi }(w)>N\}}|f'(w)|^2\frac{n_{\varphi}^{\alpha}(w)}{n_{\varphi}(w)}n_{\varphi}(w)dA(w)\\
 &\leq &\epsilon  \int_{\varphi(\D)\cap \{w: n_{\varphi }(w)>N\}}|f'(w)|^2n_{\varphi}(w)dA(w)\\
 &\leq &\epsilon  \int_{\varphi(\D)}|f'(w)|^2n_{\varphi}(w)dA(w)\\
 &\leq &\epsilon  \|C_{\varphi }\|^{2}\cdot \int_{\D}|f'(w)|^2dA(w)
\end{eqnarray*}
and
\begin{eqnarray*}
 &&\int_{\varphi(\D)\cap \{w: n_{\varphi }(w)\leq N\}}|f'(w)|^2n_{\varphi}^{\alpha}(w)dA(w)\\
  &\leq &N^{\alpha }  \int_{\varphi(\D)\cap \{w: n_{\varphi }(w)\leq N\}}|f'(w)|^2dA(w)\\
 &\leq &N^{\alpha }  \int_{\varphi(\D)}|f'(w)|^2dA(w)
\end{eqnarray*}
for all $f\in  \mathcal{D}(\D).$
Further,
\begin{eqnarray*}
&& \int_{\varphi(\D)}|f'(w)|^2n_{\varphi}^{\alpha}(w)dA(w)\\
&\leq &\epsilon \|C_{\varphi }\|^{2}\cdot \int_{\D}|f'(w)|^2dA(w) +N^{\alpha }\int_{\varphi (\D)}|f'(w)|^2dA(w)
\end{eqnarray*}
for all $f\in  \mathcal{D}(\D).$
If there is a $\delta>0$ such that
 $$ 
 \delta \int_{\D}|f'(w)|^2dA(w)\leq  \int_{\varphi(\D)}|f'(w)|^2n_{\varphi}^{\alpha}(w)dA(w)
$$
for all $f\in  \mathcal{D}(\D),$ then
\begin{eqnarray*}
\delta \int_{\D}|f'(w)|^2dA(w)&\leq & \int_{\varphi(\D)}|f'(w)|^2n_{\varphi}^{\alpha}(w)dA(w)\\
 &\leq &\epsilon \|C_{\varphi }\|\cdot \int_{\D}|f'(w)|^2dA(w)\\
 &+&N^{\alpha }\int_{\varphi (\D)}|f'(w)|^2dA(w).
\end{eqnarray*}
Choose $\epsilon_0 >0$ such that 
$
\epsilon_0 \|C_{\varphi }\|^{2}<\frac{1}{2}\delta 
$
and  $N_{0}\in \mathbb{N}$ such that 
$
\frac{n^\alpha }{n}<\epsilon_0
$
for all $n\geq N_0.$ Then 
$$
\frac{\delta }{2N_{0}^{\alpha }}\int_{\D}|f'(w)|^2dA(w)\leq \int_{\varphi (\D)}|f'(w)|^2dA(w).
$$
By Theorem KRL again, we have  $(c)\Longrightarrow (a)$ .

$(a)\Longrightarrow (b)$ is obvious since $n_{\varphi }(w)\geq 1$ for all $w\in \varphi(\D).$  To prove $(b)\Longrightarrow (a),$ assume  for arbitrary $\alpha \in (0, 1)$, there is a   $\delta >0$  and $r>0$ such that
  $$
  \int_{\varphi(\D)\cap D(z, r)}n_{\varphi}^{\alpha}(w)dA(w)\geq \delta A(D(z, r))
$$
 for all  $z\in  \D.$ Write
 $E_N=\varphi(\D)\cap D(z,r)\cap \{w:n_{\varphi }(w)>N\}$ for any $N.$ Then for any $\alpha \in (0, 1),$
\begin{eqnarray*}
 \int_{E_N}n_{\varphi }^{\alpha }(w)dA(w)&\leq &\frac{1}{N^{1-\alpha }}\int_{E_N}n_{\varphi }dA(w)\\
 &\leq &\frac{1}{N^{1-\alpha }} \int_{D(z, r)}n_{\varphi }dA(w).
\end{eqnarray*}
 Since $C_{\varphi }$ is bounded on $ \mathcal{D}(\D),$ there is a $M>0$ such that 
 $$
 \int_{D(z, r)}n_{\varphi }dA(w)\leq MA(D(z, r))
 $$ (see \cite{L1}). Hence
 $$
  \int_{E_N}n_{\varphi }^{\alpha }(w)dA(w)\leq \frac{M}{N^{1-\alpha }}A(D(z, r)).
 $$
On the other hand, 
 \begin{eqnarray*}
  \int_{\varphi(\D)\cap D(z, r)-E_N}n_{\varphi }^{\alpha }(w)dA(w)&\leq&N^{\alpha }  \int_{\varphi(\D)\cap D(z, r)-E_N}dA(w)\\
  &\leq &N^{\alpha } A(\varphi(\D)\cap D(z, r)).
\end{eqnarray*}
Thus,
 \begin{eqnarray*}
  \int_{\varphi(\D)\cap D(z, r)}n_{\varphi }^{\alpha }(w)dA(w)&=&  \int_{E_N}n_{\varphi }^{\alpha }(w)dA(w)\\
  &+& \int_{\varphi(\D)\cap D(z, r)-E_N}n_{\varphi }^{\alpha }(w)dA(w)\\
  &\leq &\frac{M}{N^{1-\alpha }}A(D(z, r))\\
  &+&N^{\alpha }   A(\varphi(\D)\cap D(z, r)).
 \end{eqnarray*}
 Choose $N_0>0$ such that 
 $
 \frac{M}{N_{0}^{1-\alpha }}<\frac{\delta}{2},
 $  
 then with the inequality $(b)$ we have that
 $$
  A(\varphi(\D)\cap D(z, r))\geq \frac{\delta }{2N_{0}^{\alpha }}A(D(z, r)).
  $$
  Hence $(b)\Longrightarrow (a).$ We complete the proof.
\end{proof}
\begin{thm}Let $\varphi $ be an  analytic self-mapping of $\D.$  Assume that $C_{\varphi }$ is a bounded composition operator on $
 \mathcal D(\D).$  If for arbitrary  $\alpha \in (0, 1)$, there is a   $\delta >0$  and $r>0$ such that
  $$
  \int_{\varphi(\D)\cap D(z, r)}n_{\varphi}^{\alpha}(w)dA(w)\geq \delta A(D(z, r))
$$
 for all  $z\in  \D.$ Or equivalently, $n_{\varphi}^{\alpha}(w)dA(w)$ is a reverse Carleson measure, 
then  $R( C_{\varphi })$  is closed.
 \end{thm}
\begin{proof}
 Since for arbitrary  $\alpha \in (0, 1)$, there is a   $\delta >0$  and $r>0$ such that
  $$
  \int_{\varphi(\D)\cap D(z, r)}n_{\varphi}^{\alpha}(w)dA(w)\geq \delta A(D(z, r))
$$
 for all  $z\in  \D,$ we know that there is a $\tilde{\delta} >0$ such that
 $$
  \int_{\varphi(\D)}|f'(w)|^2dA(w)\geq  \tilde{\delta}\int_{\D}|f'(w)|^2dA(w)
  $$
  for all $f\in \mathcal{D}(\D)$ by Proposition 2.2 and Theorem KRL.
 Thus,
  \begin{eqnarray*}
  \|C_{\varphi }f\|^2&=&|f(\varphi (0))|^2+\int_{\D}|(f\circ \varphi)'(z)|^2 dA(z)\\
 &=&|f(\varphi (0))|^2+\int_{\D}|f'(\varphi(z))|^2|\varphi'(z)|^2 dA(z)\\
  &=&|f(\varphi (0))|^2+\int_{\varphi (\D)}|f'(w)|^2n_{\varphi}(w) dA(w)\\
  &\geq &|f(\varphi (0))|^2+\int_{\varphi (\D)}|f'(w)|^2dA(w)\\
   &\geq &|f(\varphi (0))|^2+\tilde{\delta} \int_{\D}|f'(w)|^2dA(w).
   \end{eqnarray*}
It is easy to see that $R(C_{\varphi })$ has closed range. In fact, if $R(C_{\varphi})$ is not closed, 
 then there is a sequence $\{g_{m}\}\subset \mathcal D(\D)$ with $\|g_{m}\|=1$ such that 
$$
\|C_{\varphi }g_{m}\|\rightarrow 0\hskip 4mm \mbox{as}\hskip 4mm m\rightarrow  \infty .
$$
Without loss of generality, assume $g_{m}\xrightarrow{w}g.$ Then 
$$
C_{\varphi }g_{m}\xrightarrow{w}C_{\varphi }g.
$$ 
By
 $
\|C_{\varphi }g_{m}\|\rightarrow 0,$
we see that $C_{\varphi }g=0.$ Thus $g=0.$ That is, 
$
g_{m}\xrightarrow{w}0.
$ 
In particular, $g_{m}(0)\rightarrow 0 $ and $g_{m}(\varphi(0))\rightarrow 0.$ Further,
  \begin{eqnarray*}
 \|C_{\varphi }g_{m}\|^2
 &\geq &|g_{m}(\varphi (0))|^2+\tilde{\delta}\int_{\D}|g_{m}'(w)|^2dA(w)\\
  &\geq &\tilde{\delta}\int_{\D}|g_{m}'(w)|^2dA(w).
 \end{eqnarray*}
Note $\|g_{m}\|=1$ and $g_{m}(0)\rightarrow 0,$ we see that 
$$
\int_{\D}|g_{m}'(w)|^2dA(w)\rightarrow 1\hskip 5mm \mbox{as}\hskip 5mm m\rightarrow \infty .
$$ 
This contradicts to 
$
\|C_{\varphi }g_{m}\|\rightarrow 0
$
as
$m\rightarrow  \infty .
$
Hence, $R(C_{\varphi})$ is closed. This completes the proof. 
 \end{proof}
 \begin{remr}The operator $T$ on the Hilbert or Banach space $H$ is called semi-Fredholm operator if $R(T),$ the range of $T,$ is closed and at least one of $Ker  T$ and $Ker T^*$ is a finite dimensional space of $H.$ Since $Ker C_{\varphi }$ is always trivial, we see that $C_{\varphi }$ is semi-Fredholm operator if and only if $C_{\varphi }$ has closed range. \end{remr}
 \begin{cor}Let $\varphi $ be the  analytic function on $\D$ which maps  $\D$ into $\D.$  $C_{\varphi }$ is the bounded composition operator on $
 \mathcal D(\D).$  If for arbitrary  $\alpha \in (0, 1)$, there is a   $\delta >0$  and $r>0$ such that
  $$
  \int_{\varphi(\D)\cap D(z, r)}n_{\varphi}^{\alpha}(w)dA(w)\geq \delta A(D(z, r))
$$
 for all  $z\in  \D.$ Or equivalently, $n_{\varphi}^{\alpha}(w)dA(w)$ is a reverse Carleson measure, 
then $C_{\varphi }$  is a semi-Fredholm  operator.
 \end{cor}
 \begin{remr}
Proposition 2.2 seems to mean that the condition in Theorem 2.3 is also necessary. We have the following conjecture
\begin{thmc}
Let $\varphi $ be an  analytic self-mapping of $\D.$  Assume that $C_{\varphi }$ is a bounded composition operator on $
 \mathcal D(\D).$ Then $R( C_{\varphi })$  is closed if and only if  for arbitrary  $\alpha \in (0, 1)$, there is a   $\delta >0$  and $r>0$ such that
  $$
  \int_{\varphi(\D)\cap D(z, r)}n_{\varphi}^{\alpha}(w)dA(w)\geq \delta A(D(z, r))
$$
 for all  $z\in  \D.$ Or equivalently, $n_{\varphi}^{\alpha}(w)dA(w)$ is a reverse Carleson measure.
 \end{thmc}
\end{remr}

If the counting function $n_{\varphi }$ of the mapping $\varphi $  is unbounded, as long as the integrals, relative to the Carleson measure of  functions   in the unit sphere of the Dirichlet space,  are equally absolute continuous, then the reverse Carleson condition can ensure that the composition operator has  closed range. That is, we have MainTheorem.
\begin{proof}[Proof of Main Theorem] By the condition in the main Theorem,  we see easily  that $C_{\varphi }$ is bounded on $\mathcal{D}(\D).$ $(a)\Longrightarrow (c)$ and $(b)\Longrightarrow (c)$ are obvious.  Also, Theorem KRL implies that  $(b)\Longrightarrow (a)$.  It remains to show $(c)\Longrightarrow (b).$ Assume  $(c)$ holds, that is,  there is a $\delta >0$ and $r>0$ such that 
$$
\int_{\varphi(\D)\cap D(z, r)}n_{\varphi }(w)dA(w)\geq \delta A(D(z, r))
$$
for all $z\in \D.$ We are to prove that  there is a $\delta >0$ and $r>0$ such that
 $$
  A(\varphi(\D)\cap D(z,r) )\geq \delta A(D(z,r))
$$
 for all  $z\in \D$. For any $z\in \D$ and fixed $r>0, $ write 
 $$
 k_{z}(w)=\frac{1-|z|^2 }{(1-\bar{z}w)^2},
 $$
  then by the equivalence of $1-|z|^2$ and $|1-\bar{z}w|$ on $D(z, r),$ together with assumption $(c),$ we get that there exist positive constants $M_1$ , $M_2$ such that
\begin{eqnarray*}
&&\delta \leq \frac{1}{A(D(z, r))}\int_{D(z, r)}n_{\varphi }(w)dA(w)\\
&\leq &M_1\int_{D(z, r)}|k_z(w)|^2n_{\varphi }(w)dA(w)\\
&=&M_1[\int_{D(z, r)\cap \{w\in \D:n_{\varphi }(w)>k\}}|k_z(w)|^2n_{\varphi }(w)dA(w)\\
&+&\int_{D(z, r)\cap \{w\in \D:n_{\varphi }(w)\leq k\}}|k_z(w)|^2n_{\varphi }(w)dA(w)]\\
&\leq &M_1[sup_{f\in (\mathcal{D}(\D))_1}\int_{\varphi(\D)\cap \{w\in \D:n_{\varphi }(w)>k\}}|f'(w)|^2n_{\varphi }(w)dA(w)\\
&+& \int_{D(z, r)\cap \{w\in \D:n_{\varphi }(w)\leq k\}}|k_z(w)|^2n_{\varphi }(w)dA(w)]\\
& \leq &M_1sup_{f\in (\mathcal{D}(\D))_1}\int_{\varphi(\D)\cap \{w\in \D:n_{\varphi }(w)>k\}}|f'(w)|^2n_{\varphi }(w)dA(w)\\
&+&\frac{M_2}{A(D(z, r))}\int_{D(z, r)\cap \{w\in \D:n_{\varphi }(w)\leq k\}}n_{\varphi }(w)dA(w).
\end{eqnarray*}
Note that
$$
lim_{k\rightarrow \infty }sup_{f\in (\mathcal{D}(\D))_1}\int_{\varphi(\D)\cap \{w\in \D:n_{\varphi }(w)>k\}}|f'(w)|^2n_{\varphi }(w)dA(w)=0,
$$
then for $\epsilon =\frac{\delta }{2(M_1 +1)},$ there is a $k_0>0$ such that 
$$
sup_{f\in (\mathcal{D}(\D))_1}\int_{\varphi(\D)\cap \{w\in \D:n_{\varphi }(w)>k\}}|f'(w)|^2n_{\varphi }(w)dA(w)<\epsilon 
$$
for all $k\geq k_0.$ This makes that
\begin{eqnarray*}
&&M_1sup_{f\in (\mathcal{D}(\D))_1}\int_{\varphi(\D)\cap \{w\in \D:n_{\varphi }(w)>k_0\}}|f'(w)|^2n_{\varphi }(w)dA(w)\\
&+&\frac{M_2}{A(D(z, r))}\int_{D(z, r)\cap \{w\in \D:n_{\varphi }(w)\leq k_0\}}n_{\varphi }(w)dA(w)\\
&\leq &M_1\epsilon +\frac{k_0M_2}{A(D(z, r))}\int_{\varphi (\D)\cap D(z, r)}dA(w)\\
&\leq &\frac{\delta }{2}+\frac{k_0M_2}{A(D(z, r))}A(\varphi (\D)\cap D(z, r)).
\end{eqnarray*}
Hence,
$$
A(\varphi (\D)\cap D(z, r))=\int_{\varphi (\D)\cap D(z, r)}dA(w)\geq \frac{\delta }{2k_0M_2}A(D(z, r)).
$$
That is, $(c)\Longrightarrow (b).$  The proof is thus  complete.
\end{proof}
\begin{cor}
 Let $\varphi $ be an  analytic self-mapping of $\D$  and $n_{\varphi }(w)$ be bounded on $\D.$ Then the following statements are equivalent:
 \begin{enumerate}
\item[{\bf(a)}] 
 $R( C_{\varphi })$  is closed on $\mathcal{D}(\D).$
  \item[{\bf(b)}] There is a $\delta >0$ and $r>0$ such that
 $$
  A(\varphi(\D)\cap D(z,r) )\geq \delta A(D(z,r))
$$
 for all  $z\in \D$.
 \item[{\bf(c)}]  $n_{\varphi }(w)dA(w)$ is a reverse Carleson measure.
  \end{enumerate}
 \end{cor}
 \begin{proof}If $n_{\varphi}$ is bounded on $\D,$ then $C_{\varphi}$ is obviously bounded on  $\mathcal{D}(\D)$ and $n_{\varphi}$ satisfies the condition in  Theorem B. Hence $(a)\Longleftrightarrow (b)\Longleftrightarrow (c)$.
 \end{proof}

 \section{Composition operators with some special symbols}
If $\varphi $ is analytic on the closed unit disc, then we have the following
 \begin{prop}Let $H(\bar{\D})$ be the space of analytic functions on $\bar{\D},$ $\varphi \in H(\bar{\D})$ maps $\D$ into $\D.$   Then $R( C_{\varphi })=\{f\circ \varphi |f\in \mathcal D(\D)\}$, the range of $ C_{\varphi },$ is closed if and only if 
 $$
 \mathbb{T}\subset \varphi (\bar{\D}).
 $$
 \end{prop}
 \begin{proof}
 Assume $R( C_{\varphi })$ is closed, since $kerC_{\varphi }=\{0\},$ we know that there is a constant $c>0$ such that for any $f\in \mathcal D(\D),$
 $$
\| C_{\varphi }f\|\geq c\|f\|.
$$
Thus
\begin{eqnarray*}
\| C_{\varphi }f\|^2&=&\int_{\D}|(f\circ \varphi )'(z)|^2 dA(z)\\
&=&\int_{\varphi (\D)}|f'(w)|^2 n_{\varphi }(w)dA(w)\\
&\geq &c^2 \int_{\D}|f'(z)|^2 dA(z).
\end{eqnarray*}
 Note $\varphi \in H(\overline{\D}),$ we see that $C_{\varphi }$ is bounded on $\mathcal D(\D)$ since $n_{\varphi}(w)$ is bounded on $\D.$ If $\mathbb{T}-\varphi (\bar{\D})\neq \phi ,$ choose $\zeta \in \mathbb{T}-\varphi (\bar{\D}).$ Let
$$
f_{\zeta }(z)=\frac{1+\bar{\zeta}\cdot z}{2}, \hskip 5mm f_{k}(z)=f_{\zeta }^{k}(z),
$$
then $f_{\zeta }(z)$ is the peak function at $\zeta $ on $\D.$ For any open neighborhood $U$ of $\zeta ,$ it is easy to see that
\begin{eqnarray*}
&&lim_{k\rightarrow \infty }\sqrt[k]{\frac{\int_{\D-U}|f'_{k}(z)|^2dA(z)}{\int_{\D}|f'_{k}(z)|^2dA(z)}}\\
&=&\frac{max|f_{\zeta}|_{\D-U}|}{max|f_{\zeta}|_{\D}|}=max|f_{\zeta}|_{\D-U}|<1.
\end{eqnarray*}
Hence 
$$
\frac{\int_{\D-U}|f'_{k}(z)|^2dA(z)}{\int_{\D}|f'_{k}(z)|^2dA(z)}\rightarrow 0. 
$$
Note
\begin{eqnarray*}
\frac{\| C_{\varphi }f_{k}\|^2}{\|f_{k}\|^2}&=&\frac{\int_{\D}|(f_{k}\circ \varphi )'(z)|^2 dA(z)}{\int_{\D}|f'_{k}(z)|^2dA(z)}\\
&=&\frac{\int_{\varphi (\D)}|f'_{k}(w)|^2 n_{\varphi }(w)dA(w)}{\int_{\D}|f'_{k}(z)|^2dA(z)}\\
&\leq &M\frac{\int_{\varphi (\D)}|f'_{k}(w)|^2 dA(w)}{\int_{\D}|f'_{k}(z)|^2dA(z)},
\end{eqnarray*}
and there is a neighborhood $U$ of $\zeta $ such that $U\cap \varphi (\bar{\D})=\phi ,$ thus $\varphi (\bar{\D})\subset \bar{\D}-U.$ Further,
$$
\frac{\| C_{\varphi }f_{k}\|^2}{\|f_{k}\|^2}\leq M\frac{\int_{\varphi (\D)}|f'_{k}(w)|^2 dA(w)}{\int_{\D}|f'_{k}(z)|^2dA(z)}\rightarrow 0.
$$
This contradicts to $ \| C_{\varphi }f\|\geq c\|f\|. $ It shows that $\mathbb{T}\subset \varphi(\bar{\D}).$

Conversely, assume $\mathbb{T}\subset \varphi(\bar{\D}),$ we are to prove that $R(C_{\varphi})$ is closed. Assume the contrary, $R(C_{\varphi})$ is not closed, then there is a sequence $\{g_{k}\}\subset \mathcal D(\D)$ with $\|g_{k}\|=1$ such that 
$
\|C_{\varphi }g_{k}\|\rightarrow 0.
$
Without loss of generality, assume $g_{k}\xrightarrow{w}g.$ Then 
$$
C_{\varphi }g_{k}\xrightarrow{w}C_{\varphi }g.
$$
 By
 $
\|C_{\varphi }g_{k}\|\rightarrow 0,
$
we see that $C_{\varphi }g=0.$ Thus $g=0.$ That is, $g_{k}\xrightarrow{w}0.$ Further, 
$
g_{k}\rightarrow 0
$ 
uniformly on any compact subset of $\D.$

Since $\varphi \in H(\bar{\D}),$ for any $\zeta \in \mathbb{T},$ there is an open  neighborhood $U(\zeta )$ such that $n_{\varphi }(w)\geq 1$   on $U(\zeta ).$ By finite covering thoreom, there are finte $\zeta_{i}, i=1,\cdots ,m$ such that 
$$
\mathbb{T}\subset \cup_{i=1}^{m}U(\zeta_{i}),
$$
 and $n_{\varphi }(w)\geq 1$  on $U(\zeta_{i} ).$ Note $\cup_{i=1}^{m}U(\zeta_{i})$ is open, we may choose $r\in (0,1)$ such that 
$$
\overline{\D-\D_{r}}\subset \cup_{i=1}^{m}U(\zeta_{i}).
$$ 
Clearly, $g_{k}\rightarrow 0$ uniformly on $\overline{\D_{r}},$ for any $\epsilon >0,$ there is a $k_{0}$ such that 
$$
\int_{\D_{r}}|g'_{k}|^{2}(z)dA(z)<\epsilon \hskip 5mm \mbox{for}\hskip 4mm k>k_{0}.
$$
Thus for $k>k_{0},$
\begin{eqnarray*}
\|C_{\varphi }g_{k}\|^2&=&\int_{\D}|(g_{k}\circ\varphi )'|^2(z)dA(z)\\
&=&\int_{\varphi(\D)}|g'_{k}|^2(w)n_{\varphi }(w)dA(w)\\
&=&\int_{\varphi(\D)-(\D-\D_{r})}|g'_{k}|^2(w)n_{\varphi }(w)dA(w)+\int_{\D-\D_{r}}|g'_{k}|^2(w)n_{\varphi }(w)dA(w)\\
&\geq &\int_{\D-\D_{r}}|g'_{k}|^2(w)n_{\varphi }(w)dA(w)\\
&\geq &\int_{\D-\D_{r}}|g'_{k}|^2(w)dA(w)\\
&=&\int_{\D}|g'_{k}|^2(w)dA(w)-\int_{\D_{r}}|g'_{k}|^2(w)dA(w)\\
&\geq &\int_{\D}|g'_{k}|^2(w)dA(w)-\epsilon \\
&=&\|g_{k}\|^2-\epsilon\\
&=&1-\epsilon .
\end{eqnarray*}
This contradicts to $
\|C_{\varphi }g_{k}\|\rightarrow 0.
$ It shows that  $R(C_{\varphi})$ must be closed.
 \end{proof}
 B. Hou and Ch.L. Jiang prove recently a similar result in the case of weighted Hardy space  of polynomial growth (see \cite{HJ}).  In general, $\mathbb{T}\subset \overline{\varphi(\D)}$ is not enough to ensure that $C_{\varphi }$ has a closed range on $\mathcal{D}(\D)$ even if $n_{\varphi } $ is bounded on $\D.$ The following example illustrates this conclusion.
\begin{exe}
As shown figure 1:
 \begin{figure}[h]
\centering
\includegraphics[height=50pt,width=130pt]{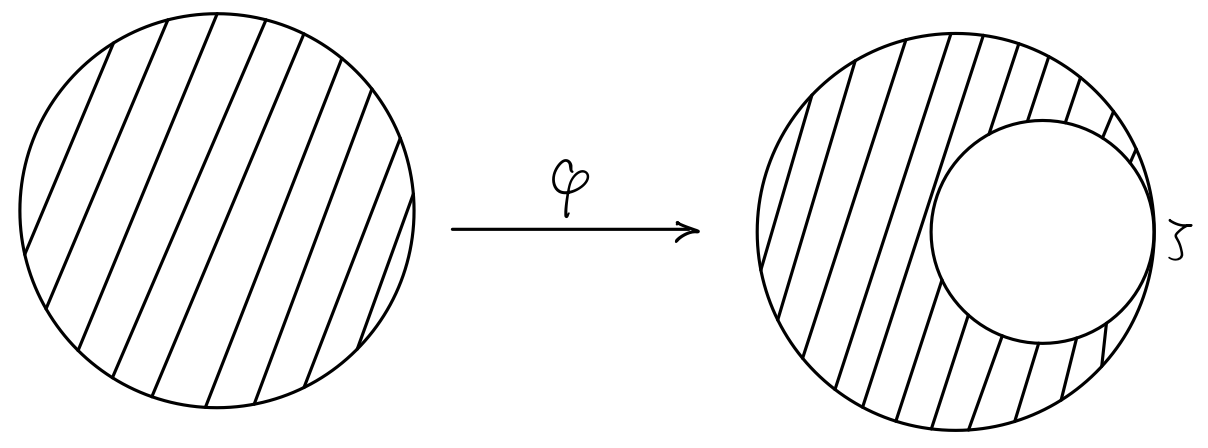}
\caption{}
\end{figure}\\
 Let $\zeta \in \mathbb{T} $ and $\D_{\zeta }$  is  the disk inside the unit disk and is tangent to $\mathbb{T}$  at point $\zeta .$ Write
$\Omega =\D-\overline{\D_{\zeta }},$ $\varphi $ is the Riemann map from $\D$ to $\Omega \subset \D.$  Then $C_{\varphi }$  is a contraction operator on $\mathcal{D}(\D)$ since $n_{\varphi }(w)\leq 1$ for all $w\in \D.$ By Corollary 2.6, we know that 
$R(C_{\varphi })$ is not closed,  although  $\mathbb{T}\subset \overline{\varphi(\D)}.$
\end{exe}
 
 
\bigskip
\bigskip
 \noindent


\begin{thebibliography}{99}
\bibitem{AF} J.R. Akeroyd and S.R. Fulmer, Closed-range composition operators on weighted Bergman spaces, {\sl Integr. Equ. Oper. Theory,} 72 (2012), 103–114.
\bibitem{CTW} J. A. Cima, J. Thompson, and W. R. Wogen, On some properties of composition operators, {\sl Indiana Univ. Math. J. }, 24 (1974),215-220.
 \bibitem{GYZ} P. Ghatage, J. Yan, and D. Zheng, Composition operators with closed range on the Bloch space,  {\sl Proc. Amer. Math. Soc.} ,129(7) (2000), 2039-2044.
\bibitem{GZZ} P. Ghatage, D. Zheng, and N. Zorboska, Sample sets and closed range composition operators on the Bloch space, {\sl Proc. Amer. Math. Soc.} ,133(5) (2004), 13711377.
\bibitem{HJ}B.  Hou and CH.L.  Jiang, Composition operators on weighted Hardy spaces of polynomial growth, Preprint.
\bibitem{KW} N. Kenessey and J. Wengenroth, Composition operators with closed range for smooth injective symbols $R \rightarrow  R^d$, {\sl J. Funct. Anal. } ,260 (2011), 2997–3006.
\bibitem{KK} H. Keshavarzi  and B. Khani-Robati, On the closed range composition and weighted composition operators, {\sl Commun. Korean Math. Soc. } ,35 (1) (2020),  217-227.
\bibitem{L} D. H. Luecking, Inequalities on Bergman spaces, {\sl Illinois J. Math. },25 (1981), 1-11.
\bibitem{L2} D. H. Luecking, Closed ranged restriction operators on weighted Bergman spaces, {\sl Pacific J. Math.} ,110(1) (1984), 145–160.
\bibitem{L3} D. H. Luecking, Forward and reverse Carleson inequalities for functions in Bergman spaces and their derivatives, {\sl Amer. J. Math.},107 (1985), 85–111.
\bibitem{L1} D. H. Luecking, Bounded composition operators with closed range on the Dirichlet space,  {\sl Proc. Amer. Math. Soc.},128(4) (1999), 1109–1116.
\bibitem{M} T.  Mengestie, Closed range weighted composition operators and dynamical sampling, {\sl J.  Math.  Anal.  Appl.}, 515 (2022), 1-11.
\bibitem{P} A. Przestacki, Characterization of composition operators with closed range for one-dimensional smooth symbols, {\sl J. Funct. Anal. }, 266 (2014) 5847–5857.
\bibitem{Y} R. Yoneda, Composition operators on the weighted Bloch space and the weighted Dirichlet spaces, and BMOA with closed range, {\sl Com. Var. Ell. Eq.}, 63(5) (2017), 730-747.
\bibitem{Z} K.H.Zhu, Operator theory in function spaces, {\sl Math. Sur. Mono.}, 138 (2007).
\bibitem{Zor} N. Zorboska, Composition operators on weighted Dirichlet spaces, {\sl Proc. Amer. Math. Soc.}, {\bf 126} (1998), 2013--2023.
\bibitem{Zo} N. Zorboska, Composition operators with closed range, {\sl Trans. Amer. Math. Soc.}, {\bf 344}
(1994), 791--801.
\bibitem{Zo1} N. Zorboska, On the closed range problem for composition operators on the Dirichlet space, {\sl Concr. Oper.}, 6 (2019), 76-81.
\end{thebibliography}
\end{document}